\documentclass[11pt]{article}
\usepackage{amssymb}
\usepackage{mathrsfs}
\usepackage{amsthm}
\usepackage{amsmath}
\usepackage{url}
\usepackage{color}
\usepackage[dvipdfmx]{graphicx}
\usepackage{asymptote}
\title{Non-separability of the Lipschitz distance}
\author{Kohei Suzuki \vspace{2mm} \\ {\it \small Department of Mathematics, Faculty of Science} \\ {\it \small Kyoto University} 
   \\ {\it \small Kyoto, 606-8502, Japan}
   \vspace{3mm} \\
            Yohei Yamazaki \vspace{2mm} \\ {\it \small Department of Mathematics, Faculty of Science} \\ {\it \small Kyoto University} 
   \\ {\it \small Kyoto, 606-8502, Japan}} 

\makeatletter
\newcommand{\subjclass}[2][2010]{%
  \let\@oldtitle\@title%
  \gdef\@title{\@oldtitle\footnotetext{#1 \emph{Mathematics subject classification.} #2}}%
}
\newcommand{\keywords}[1]{%
  \let\@@oldtitle\@title%
  \gdef\@title{\@@oldtitle\footnotetext{\emph{Key words and phrases.} #1.}}%
}
\makeatother

\subjclass{Primary 53C23; Secondary 53C20.}
\keywords{Lipschitz convergence, compact metric spaces, separability}

\newtheorem{theorem}{Theorem}

\newtheorem{lemma}[theorem]{Lemma}
\newtheorem{proposition}[theorem]{Proposition}

\theoremstyle{definition}
\newtheorem{definition}[theorem]{Definition}
\newtheorem{remark}[theorem]{Remark}

\renewcommand{\eqref}[1]{(\ref{#1})}

\renewcommand{\bigskip}{\vspace{0.3cm}}
\renewcommand{\l}{\left}
\renewcommand{\r}{\right}

\newcommand{\e}{\epsilon}

\newcommand{\R}{{\mathbb R}}

\newcommand{\Z}{{\mathbb Z}}

\newcommand{\ZP}{\Z_{>0}}

\date{}

\begin{document}
\maketitle
\begin{abstract}
Let $X$ be a compact metric space and $\mathcal M_X$ be the set of isometry classes of compact metric spaces $Y$ such that the Lipschitz distance $d_L(X,Y)$ is finite.
We show that $(\mathcal M_X, d_L)$ is not separable when $X$ is a closed interval, or an infinite union of shrinking closed intervals.   
\end{abstract}

\section{Introduction}
For compact metric spaces $(X,d_X)$ and $(Y,d_Y)$, {\it the Lipschitz distance} $d_{L}(X,Y)$ is defined to be the infimum of $\e\ge 0$ such that an $\e$-isometry $f: X \to Y$ exists.
Here a bi-Lipschitz homeomorphism $f:X\to Y$ is called {\it an $\e$-isometry} if 
\begin{align*}
	|\log {\rm dil}(f)| + |\log {\rm dil}(f^{-1})| \le \e,
\end{align*}
where ${\rm dil}(f)$ denotes the smallest Lipschitz constant of $f$, called {\it the dilation of $f$}:
$${\rm dil}(f)=\sup_{\substack{x,y\in X \\ x\neq y}}\frac{d_Y(f(x), f(y))}{d_X(x,y)}.$$

Let $\mathcal M$ be the set of isometry classes of compact metric spaces. It is well-known that $(\mathcal M, d_L)$ is a complete metric space. See, e.g., \cite[Appendix A]{S} for the proof of the completeness and see, e.g., \cite{BBI01, Gro99} for details of the Lipschitz distance.

Then the following question arises:
\begin{description}
\item[(Q)]:  Is the metric space $(\mathcal M, d_L)$ separable?
\end{description}
The answer is {\it no}, which can be seen easily by the following facts:
 \begin{description}
\item[(a)]if $d_L(X,Y)<\infty$, the Hausdorff dimensions of $X$ and $Y$ must coincide;
\item[(b)]for any non-negative real number $d$, there is a compact metric space $X$
whose Hausdorff dimension is equal to $d$. 
\end{description}
See, e.g., \cite[Proposition 1.7.19]{BBI01} for (a) and \cite{SS} for (b).

The fact (b) indicates that $(\mathcal M, d_L)$ is too big to be separable.  Then we change the question (Q) to the following more reasonable one (Q'):
For a compact metric space $X$, let $\mathcal M_X$ be the set of isometry classes of compact metric spaces $Y$ such that $d_L(X,Y)<\infty$. Any elements of $\mathcal M_X$ have a common Hausdorff dimension by (a). Then the following question arises:
\begin{description}
\item[(Q')]:  Is the metric space $(\mathcal M_X, d_L)$ separable?
\end{description}

The main results of this paper give the negative answer for this question for several $X$. To be more precise, we give two examples for $X$ such that $(\mathcal M_X, d_L)$ is not separable:
\begin{description} 
	\item[(i)] Infinite unions of shrinking closed intervals with zero
    \[ \{ 0 \} \cup \bigcup_{n=1}^{\infty}\l[\frac{1}{2^n}, \frac{1}{2^{n}}+\frac{1}{2^{n+1}}\r]; \]
	\item[(ii)] Closed interval $[0,1]$.
\end{description}
We would like to stress that $(\mathcal M_X, d_L)$ becomes non-separable even when $X$ are the above elementary cases.
We note that the non-separability of the first example follows from the non-separability of the second example. The first example, however, is easier to show the non-separability than the second example.

The present paper is organized as follows: In the first section, we show that the set of isometry classes of the infinite unions of shrinking closed intervals with zero is not separable. In the second section, we show that the set of isometry classes of 
the closed interval is not separable.

\section{The first example}

Let $\Z_{>0} = \{ n \in \Z: n>0\}$ denote the set of positive integers. For $n,m \in \ZP$, let $I(n,m)$ be an interval in $\R$ defined as follows: 
$$I(n,m)=\l[\frac{1}{2^n}, \frac{1}{2^n}+\frac{1}{2^{n+m}}\r].$$ 
For each $u=(u_n)_{n \in \ZP} \in \{1,2\}^{\ZP}$, we define the following subset in $\R$: 
\begin{align} \label{eq: Xu} 
X_{u} = \{0\} \cup \bigcup_{n=1}^{\infty} I(n,{u_n}).
\end{align}
We equip $X_{u}$ with the usual Euclidean metric in $\R$: 
\[d(x,y)=|x-y|, \quad x,y \in X_{u}.\]
Then it is easy to check that $(X_{u},d)$ is a compact metric space.

Let $\mathbf 1=(1,1,1,...) \in \{1,2\}^{\ZP}$ denote the element in $\{1,2\}^{\ZP}$ such that all components are equal to one.  Let $X_{\mathbf 1}$ be the set defined in \eqref{eq: Xu} for the element $\mathbf 1$.
Let $\mathcal M_{X_{\mathbf 1}}$ denote the set of isometry classes of compact metric spaces $X$ whose Lipschitz distances from $X_{\mathbf 1}$ are finite, that is, $d_L(X, X_\mathbf 1)<\infty$. Then we have the following result:
\begin{theorem} \label{thm: Xu}
$(\mathcal M_{X_{\mathbf 1}}, d_L)$ is not separable.
\end{theorem}
\proof
It is enough to find a certain discrete subset $\mathbb X \subset \mathcal M_{X_\mathbf 1}$ with the continuous cardinality. 
We introduce a subset $\mathbb X \subset \mathcal M$, which is the set of isometry classes of all $X_u$ for $u \in \Z_{>0}$:
\[ \mathbb{X}=\{ (X_{u},d) : u \in \{1,2\}^{\Z_{>0}}\}/\text{isometry}.\]
It is clear that the cardinality of $\mathbb X$ is continuum.
We show that $\mathbb{X} \subset \mathcal{M}_{X_{\mathbf 1}}$ and $\mathbb{X}$ is discrete (i.e., every point in $\mathbb X$ is isolated).

We first show that $\mathbb{X} \subset \mathcal{M}_{X_{\mathbf 1}}$.
For $u=(u_n)_{n \in \ZP} \in \{1,2\}^{\Z_{>0}}$ and $v=(v_n)_{n \in \ZP} \in \{1,2\}^{\Z_{>0}}$, let $f_{u,v}$ be a function from $X_{u}$ to $X_{v}$ defined by
\begin{align*}
f_{u,v}(x)=
\begin{cases} \displaystyle
0 \quad (x=0),\\
\displaystyle \frac{2^{u_n}}{2^{v_n}}\l(x-\frac{1}{2^n}\r)+ \frac{1}{2^n} \quad \bigl(x \in I({n},{u_n}) \bigr).
\end{cases}
\end{align*}
Then $f_{u,v}$ is a bi-Lipschitz continuous function from $X_{u}$ to $X_{v}$ and for $x,y \in X_{u}$,
\[\frac{1}{2}|x-y| \leq |f_{u,v}(x)-f_{u,v}(y)|\leq 2|x-y|.\]
Therefore the Lipschitz distance between $X_u$ and $X_v$ is bounded by
\begin{align*}
d_L(X_{u},X_{v}) \leq 2\log 2 \quad \text{for any $u,v \in \ZP$}.
\end{align*}
Thus we have that $\mathbb{X} \subset \mathcal{M}_{X_{\mathbf 1}}$.

Second we show that $\mathbb{X}$ is discrete:
\begin{lemma}\label{CE}
Let $X_{u}, X_{v} \in \mathbb{X}$.
If $d_L(X_{u},X_{v})<\log 2$, then $u=v$.
\end{lemma}
\begin{proof}
Let $u=(u_n)_{n \in \ZP} \in \{1,2\}^{\Z_{>0}}$ and $v=(v_n)_{n \in \ZP} \in \{1,2\}^{\Z_{>0}}$. 
We show that $u_n=v_n$ for all $n \in \ZP$. 
By the assumption $d_L(X_{u},X_{v})<\log 2$, there exists a bi-Lipschitz function $f:X_{u} \to X_{v}$ such that 
\begin{equation}\label{ex1-ass}
|\log \mbox{dil}(f)|+|\log \mbox{dil}(f^{-1})|<\log 2.
\end{equation}
Since $f$ is homeomorphic, any intervals must be mapped to intervals by $f$. 
That is, there exists a bijection $P : \Z_{>0} \to \Z_{>0}$ as $n \mapsto P(n)$ such that 
\[f(I(n,{u_n}))=I({P(n)},{v_{P(n)}}).\]

To show $u_n=v_n$ for all $n \in \ZP$, we have two steps: 
\begin{description}
	\item[(i)] $n+u_n = P(n) + v_{P(n)}$;
	\item[(ii)] $P(n)=n$.
\end{description}

We start to show (i) by contradiction.
Assume there exists $n_0 \in \Z_{>0}$ such that $n_0+u_{n_0} \neq P(n_0)+v_{P(n_0)}$.
Since $f|_{I({n_0},{u_{n_0}})}$ is homeomorphic, the endpoints of $I({n_0},{u_{n_0}})$ must be mapped to the endpoints of $I({P(n_0)},{v_{P(n_0)}})$ by $f|_{I({n_0},{u_{n_0}})}$.
Therefore 
\[ 
\l| f\l(\frac{1}{2^{n_0}}\r) - f\l(\frac{1}{2^{n_0}}+ \frac{1}{2^{n_0+u_{n_0}}} \r) \r| = \frac{1}{2^{P(n_0)+v_{P(n_0)}}}.
\]
Thus the dilation of $f$ is at least bigger than 
\begin{align*}
{\rm dil}(f) 
&\ge \frac{| f(1/2^{n_0}) - f(1/2^{n_0}+1/2^{n_0+u_{n_0}})|}{|1/2^{n_0}-(1/2^{n_0}+1/2^{n_0+u_{n_0}})|}
\\
& = \frac{1}{2^{P(n_0)+v_{P(n_0)}-(n_0+u_{n_0})}}.
\end{align*}
By  the assumption of  $n_0+u_{n_0} \neq P(n_0)+v_{P(n_0)}$, we have that ${\rm dil}(f) \ge 2$ or ${\rm dil}(f^{-1}) \ge 2$.
This implies 
$$|\log\mbox{dil}(f)|\geq \log 2 \quad \text{or} \quad  |\log\mbox{dil}(f^{-1})|\geq \log 2. $$
This contradicts the inequality \eqref{ex1-ass}.
Hence we have $n+u_n = P(n) + v_{P(n)}$ for all $ n \in \Z_{>0}$.

We start to show (ii) by contradiction. 
Assume there exists $n_0 \in \Z_{>0}$ such that $P(n_0)\neq n_0$.
Let us define
\[ n_*= \min \{ n \in \Z_{>0}| P(n) \neq n\}.\]
Then $P(n_*)>n_*$ by definition.
Since we know that $n+u_n=P(n)+v_{P(n)}$ by the first step (i), and that $u_n$ and $v_{P(n)}$ are in $\{1,2\}$, thus the possibility of values of $P(n)$ is that $P(n)=n-1$, $n$ or $n+1$.
This implies that $P(n_*)=n_*+1$, $P(n_*+1)=n_*$ and $P(n_*+2) =n_*+2,$ or $n_*+3$. 
Since the endpoints of intervals must be mapped to the endpoints of intervals by $f$, the possibility of values of $f( 1/2^{n_*+1})$ and $ f( 1/2^{n_*+2})$ is 
\[ f\l( \frac{1}{2^{n_*+1}}\r)= \frac{1}{2^{n_*}}, \quad \text{or} \quad \frac{1}{2^{n_*}}+\frac{1}{2^{n_*+v_{n_*}}}, \]
and
\begin{align*}
 f\l( \frac{1}{2^{n_*+2}}\r) &= \frac{1}{2^{n_*+2}},  \ \frac{1}{2^{n_*+2}} + \frac{1}{2^{n_*+2+v_{(n_*+2)}}}, \ \frac{1}{2^{n_*+3}},  
 \\
 &\text{or}  \ \frac{1}{2^{n_*+3}} + \frac{1}{2^{n_*+3+v_{(n_*+3)}}}.
\end{align*}
Thus, by noting $v_{P(n_*+2)} \in \{1,2\}$, we have the following estimate: 
\begin{align*}
& \l|f\l( \frac{1}{2^{n_*+1}}\r)-f\l( \frac{1}{2^{n_*+2}}\r)\r| 
\\
&\geq \l|\frac{1}{2^{n_*}}- \l(\frac{1}{2^{n_*+2}} + \frac{1}{2^{n_*+2+v_{P(n_*+2)}}}\r)\r|
\\
&\ge \frac{5}{2} \frac{1}{2^{n_*+2}}.
\end{align*}
This shows $|\log\mbox{dil}(f)|\geq \log (5/2)$ and contradicts the inequality \eqref{ex1-ass}.
Hence we have $P(n)=n$ for all $n \in \Z_{>0}$.

By the above two steps, we have that $u_n=v_n$ for all $n \in \ZP$, and we have completed the proof of Lemma \ref{CE}.
\end{proof}
We resume the proof of Theorem \ref{thm: Xu}.

{\it Proof of Theorem \ref{thm: Xu}.} By using Lemma \ref{CE}, we know that $(\mathbb X, d_L)$ is discrete. 
Since the cardinality of $\mathbb X$ is continuum and $\mathbb X \subset \mathcal M_{X_\mathbf 1}$, we have that $(\mathcal M_{X_\mathbf 1}, d_L)$ is not separable. We have completed the proof.
\qed

\section{The second example}
In this section, we show the non-separability of $\mathcal M_{[0,1]}$:
\begin{theorem} \label{thm: Int}
 The metric space $(\mathcal M_{[0,1]},d_L)$ is not separable.
\end{theorem}
{\it Proof.}
It is enough to find a certain discrete subset $\mathbb Y \subset \mathcal M_{[0,1]}$ with the continuous cardinality.

Define two subsets,  {\it flat parts} $J(n,0)$, and {\it pulse parts} $J(n,1)$ in $\R^2$:
\begin{itemize}
\item Flat part: for $n \in \ZP$,
\begin{align*}
J(n,0)=&\l[ \frac{1}{2^n},\frac{1}{2^{n-1}} \r]\times \{0\},
\end{align*}
\item Pulse part: for $n \in \ZP$,
\begin{align*}
J(n,1)=&\l[ \frac{3}{2^{n+1}},\frac{1}{2^{n-1}} \r]\times \{0\}\\
&\cup \l\{ \l(x,\frac{3}{2^{n+1}}-x \r) : \frac{5}{2^{n+2}} \leq x \leq \frac{3}{2^{n+1}} \r\}\\
&\cup \l\{\l(x,x-\frac{1}{2^n}\r): \frac{1}{2^n} \leq x \leq \frac{5}{2^{n+2}}\r\}.
\end{align*}
\end{itemize}
See the figures below:
\begin{figure}[htbp]
\begin{center}
\includegraphics{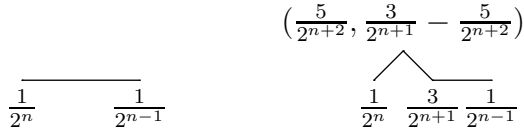}
\caption{The left is $J(n,0)$ and the right is $J(n,1)$.}
\label{picture.1}
\end{center}
\end{figure}
\newpage
\noindent For each $u=(u_n)_{n \in \ZP} \in \{0,1\}^{\ZP}$, let $Y_u$ be a subset in $\R^2$ as an infinite union of flat parts and pulse parts with the origin:
\[Y_{u} = \{(0,0)\} \cup \bigcup_{n=1}^{\infty} J(n,u_n) \subset \R^2.\]
See the figure below:
\begin{figure}[h]
\begin{center}
\includegraphics{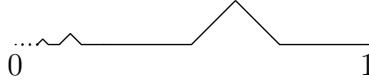}
\caption{A picture of $Y_u$.}
\label{picture.2}
\end{center}
\end{figure}

\noindent We equip $Y_u$ with the usual Euclidean distance in $\R^2$:
\begin{align} \label{equal: Euc} 
d((x_1,x_2),(y_1,y_2))=((x_1-y_1)^2+(x_2-y_2)^2)^{1/2}.
\end{align}
It is easy to check that $(Y_{u},d)$ is a compact metric space. 
Let $\mathbb Y$ be the set of isometry classes of $Y_u$ for all $u \in \{0,1\}^{\ZP}$:
\[\mathbb{Y}=\{Y_{u}:u \in \{0,1\}^{\Z_{>0}}\}/\text{isometry}.\]

Now we show that $\mathbb Y \subset \mathcal M_{[0,1]}$.
For $u \in \{0,1\}^{\Z_{>0}}$, let $f_{u}$ be the projection from $Y_u$ to $[0,1]$ such that $x=(x_1,x_2) \mapsto x_1$.
Then it is easy to see that $f_{u}$ is bi-Lipschitz continuous and, for $x,y\in Y_u$, 
\begin{align} \label{ineq: Lip2}
\frac{1}{\sqrt{2}}d(x,y) \leq |f_{u}(x)-f_{u}(y)| \leq d(x,y).
\end{align}
Therefore the Lipschitz distance between $[0,1]$ and $Y_u$ is bounded by
\[d_L([0,1],Y_{u})\leq \frac{1}{2}\log 2 \quad \text{$\forall u \in \{0,1\}^{\ZP}$}.\]
Thus we have $\mathbb Y \subset \mathcal M_{[0,1]}$.

Now we show that $\mathbb Y$ is discrete:
\begin{lemma}\label{CE2}
Let $Y_{u},Y_{v} \in \mathbb{Y}.$
If $$d_L(Y_{u},Y_{v} )<\frac{\log (\sqrt{2}+1)-\log \sqrt{5}}{2},$$ then $u=v$.
\end{lemma}
\begin{proof}
Let $u=(u_n)_{n \in \ZP} \in \{1,2\}^{\Z_{>0}}$ and $v=(v_n)_{n \in \ZP} \in \{1,2\}^{\Z_{>0}}$. 
We show that $u_n=v_n$ for all $n \in \ZP$. 
By the assumption,  there exists a bi-Lipschitz function $f$ from $Y_{u}$ to $Y_{v}$ such that
\begin{equation}\label{dil-2}
 |\log \mbox{\rm dil}(f)|+|\log \mbox{\rm dil}(f^{-1})| <\frac{\log (\sqrt{2}+1)-\log \sqrt{5}}{2}.
\end{equation}
Let us define a subset in $\ZP$ as follows: $$ P_u=\{n \in \ZP: u_n=1\}. $$
Without loss of generality, we may assume that $P_u$ is not empty. That is, $Y_u$ has at least one pulse.
The pulse part $J(n, u_n)$ of $Y_u$ for $n \in P_u$ is called {\it $n$-pulse of $Y_u$}. We note that, by the definition of the pulse parts, the peak of the $n$-pulse is attained at $5/2^{n+2}$ in $x$-axis.

It is enough for the desired result to show that $P_u=P_v$. We show that there is a bijection $F: P_u \to P_v$ such that $F(n)=n$. 
To show this, we have the following three steps:
\begin{description}
	\item[(i)] The first step: for $n \in P_u$, 
	\begin{align*} 
	f_{v}\circ f \circ f^{-1}_{u}\l ( \frac{5}{2^{n+2}}\r) \in & \l\{   \frac{5}{2^{m+2}}: m \in \Z_{>0} \r\} 
	\\
	 & \cup \l\{ \frac{3}{2^{m+1}}: m \in \Z_{>0}\r\} 
	 \\
	 & \cup \l\{ \frac{1}{2^m}:m \in \Z_{>0} \r\} .
	\end{align*}
	\item[(ii)] The second step: for $n \in P_u$, 
	\begin{align*}
	f_{v}\circ f \circ f^{-1}_{u}\l ( \frac{5}{2^{n+2}}\r) \notin & \l\{ \frac{3}{2^{m+1}}: m \in \Z_{>0}\r\} 
	\\
	& \cup \l\{ \frac{1}{2^m} : m\in \Z_{>0} \r\} .
	\end{align*}
	\item[(iii)] The third step: for $n \in P_u$, 
	\[ f_{v}\circ f \circ f^{-1}_{u}\l ( \frac{5}{2^{n+2}}\r) =\frac{5}{2^{n+2}}\quad \text{and} \quad v_n=1 .\]
\end{description}
In fact, if we show the above three statements, each maximizers $5/2^{n+2}$ of $n$-pulses of $Y_u$ are mapped to the maximizers $5/2^{n+2}$ of $n$-pulses of $Y_v$ by $f_{v}\circ f \circ f^{-1}_{u}$.
This correspondence of $n$-pulses defines the map $F: P_u \to P_v$ such that $F(n)=n$.

The proof of the all three steps (i)-(iii) are governed by the same scheme: 
\begin{description}
\item[(A)] Assume that the statements do not hold (proof by contradiction);
\item[(B)] Estimate lower bounds of the dilations of $f$ and $f^{-1}$; 
\item[(C)] The lower bounds obtained in (B) contradict the inequality \eqref{dil-2}.
\end{description}

We start to show the first step (i). Since $f$ is  homeomorphic, the maximizer  $5/2^{n+2}$ of the pulse cannot be mapped to the endpoints of $[0,1]$ by $f_{v} \circ f \circ f^{-1}_{u}$. 
Assume that, for some $n \in P_u$,
\begin{align*}
f_{v}\circ f \circ f^{-1}_{u}\l ( \frac{5}{2^{n+2}}\r) \notin & \l\{ \frac{5}{2^{m+2}} :m\in \Z_{>0} \r\}
\\
& \cup \l\{ \frac{3}{2^{m+1}} :m \in \Z_{>0}\r\} 
\\
& \cup \l\{ \frac{1}{2^m} :m \in \Z_{>0} \r\},
\end{align*}
and prove (i) by contradiction. By the continuity of $f$, there exists $0<\delta< 1/2^{n+3}$ such that, for any $x \in [5/2^{n+2}-\delta, 5/2^{n+2}+\delta]$, we have
\begin{align*}
f_{v}\circ f \circ f^{-1}_{u}\l ( x\r)  \notin & \l\{ \frac{5}{2^{m+2}} :m \in \Z_{>0} \r\}
\\
& \cup \l\{ \frac{3}{2^{m+1}}:m\in \Z_{>0}\r\} 
\\
& \cup \l\{ \frac{1}{2^m} :m \in \Z_{>0} \r\} .
\end{align*}
Therefore we have 
\[\begin{split}
&d \l(f\circ f^{-1}_{u}\l(\frac{5}{2^{n+2}}-\delta\r),f\circ f^{-1}_{u}\l(\frac{5}{2^{n+2}}+\delta\r) \r)\\
=&d \l(f\circ f^{-1}_{u}\l(\frac{5}{2^{n+2}}-\delta\r),f\circ f^{-1}_{u}\l(\frac{5}{2^{n+2}}\r) \r)\\
& +d \l(f\circ f^{-1}_{u}\l(\frac{5}{2^{n+2}}\r),f\circ f^{-1}_{u}\l(\frac{5}{2^{n+2}}+\delta\r) \r).
\end{split}\]
Here we use the fact that the three points $f\circ f^{-1}_{u}(5/2^{n+2}-\delta)$, $f\circ f^{-1}_{u}(5/2^{n+2})$ and $f\circ f^{-1}_{u}(5/2^{n+2}+\delta)$ are on the same line.
By using the inequality \eqref{ineq: Lip2}, the dilation of $f$ is estimated as follows:
\[\begin{split}
\mbox{dil}(f)
&\geq \frac{d \Bigl(f\circ f^{-1}_{u}\l(\frac{5}{2^{n+2}}-\delta\r),f\circ f^{-1}_{u}\l(\frac{5}{2^{n+2}}+\delta\r) \Bigr)}{d\Bigl( f^{-1}_{u}\l(\frac{5}{2^{n+2}}-\delta\r), f^{-1}_{u}\l(\frac{5}{2^{n+2}}+\delta\r) \Bigr)}\\
& =  \frac{d \Bigl( f\circ f^{-1}_{u}\l(\frac{5}{2^{n+2}}-\delta\r),f\circ f^{-1}_{u}\l(\frac{5}{2^{n+2}}\r) \Bigr)}{2\delta}
\\
&\quad +\frac{d \Bigl(f\circ f^{-1}_{u}\l(\frac{5}{2^{n+2}}\r),f\circ f^{-1}_{u}\l(\frac{5}{2^{n+2}}+\delta\r) \Bigr)}{2\delta}
\\
&\geq \frac{d \Bigl( f^{-1}_{u}\l(\frac{5}{2^{n+2}}-\delta\r), f^{-1}_{u}\l(\frac{5}{2^{n+2}}\r) \Bigr) }{2\delta \mbox{ dil}(f^{-1})}
\\
& \quad + \frac{d \Bigl( f^{-1}_{u}\l(\frac{5}{2^{n+2}}\r), f^{-1}_{u}\l(\frac{5}{2^{n+2}}+\delta\r) \Bigr) }{2\delta \mbox{ dil}(f^{-1})}\\
&= \frac{\sqrt{2}}{\mbox{dil}(f^{-1})}.
\end{split}\]
In the above last line, we just calculated the distance following the Euclidean distance \eqref{equal: Euc} in the $n$-pulse $J(n,1)$.
This implies that $\mbox{dil}(f)\geq 2^{\frac{1}{4}}$, or $\mbox{dil}(f^{-1}) \geq 2^{\frac{1}{4}}$. Thus we have
\[d_L(Y_{u},Y_{v}) \geq \frac{\log 2}{4}.\]
This contradicts the inequality \eqref{dil-2}.
Therefore we have, for any $n \in P_u$,
\begin{align*}
f_{v}\circ f \circ f^{-1}_{u}\l ( \frac{5}{2^{n+2}}\r) \in & \l\{ \frac{5}{2^{m+2}}: m  \in \Z_{>0} \r\}
\\
&\cup \l\{ \frac{3}{2^{m+1}} :m \in \Z_{>0}\r\} 
\\
&\cup \l\{ \frac{1}{2^m} :m  \in \Z_{>0} \r\}.
\end{align*}

We start to show the second step (ii) by contradiction. Assume that, for some $n \in P_u$, 
$$f_{v}\circ f \circ f^{-1}_{u}\l ( \frac{5}{2^{n+2}}\r) \in \l\{ \frac{1}{2^m} :m  \in \Z_{>0} \r\}.$$
Then there exists $n_1\in \Z_{>0}$ such that
\begin{align} \label{eq: step2}
f_{v}\circ f \circ f^{-1}_{u}\l ( \frac{5}{2^{n+2}}\r) = \frac{1}{2^{n_1}}.
\end{align}
By the same argument as the first step (i), we can obtain $v_{n_1}=1$, that is, the $n_1$-pulse exists in $Y_v$.
By the continuity of $f$, there exists $0< \delta < 1/2^{n_1+3}$ such that 
\begin{align*}
f_{u} \circ  f^{-1} \circ f^{-1}_{v} \l( \l[\frac{1}{2^{n_1}}-\delta,\frac{1}{2^{n_1}} \r) \r) \subset & \l(\frac{5}{2^{n+2}}-\frac{1}{2^{n+3}},\frac{5}{2^{n+2}} \r) 
\\
& \cup \l(\frac{5}{2^{n+2}},\frac{5}{2^{n+2}}+\frac{1}{2^{n+3}} \r),
\end{align*}
and 
\begin{align*}  f_{u} \circ  f^{-1} \circ f^{-1}_{v} \l( \l(\frac{1}{2^{n_1}},\frac{1}{2^{n_1}} +\delta \r] \r) \subset & \l(\frac{5}{2^{n+2}}-\frac{1}{2^{n+3}},\frac{5}{2^{n+2}} \r) 
\\
& \cup \l(\frac{5}{2^{n+2}},\frac{5}{2^{n+2}}+\frac{1}{2^{n+3}} \r). 
\end{align*}
Noting the definition of $(Y_v,d)$, for $ x \in [\frac{1}{2^{n_1}}-\delta,\frac{1}{2^{n_1}} )$ and $y \in (\frac{1}{2^{n_1}},\frac{1}{2^{n_1}} +\delta ]$, we have
\[\begin{split}
d\bigl(f^{-1}_{v}(x),f^{-1}_{v}(y)\bigr)&=\l(|x-y|^2+\l|y-\frac{1}{2^{n_1}}\r|^2\r)^{1/2},\\
d\l(f^{-1}_{v}(x),f^{-1}_{v}\l(\frac{1}{2^{n_1}}\r)\r)&=\l|x-\frac{1}{2^{n_1}}\r|,\\
d\l(f^{-1}_{v}(y),f^{-1}_{v}\l(\frac{1}{2^{n_1}}\r)\r)&=\sqrt{2}\l|y-\frac{1}{2^{n_1}}\r|.
\end{split}
\]
Since $|x-y|=|x-2^{-n_1}|+ |y-2^{-n_1}|$, we have the following inequality:
\begin{align} \label{ineq: triangle1}
&d\l(f^{-1}_{v}(x),f^{-1}_v\l(\frac{1}{2^{n_1}}\r)\r)+d\l(f^{-1}_{v}(y),f^{-1}_v\l(\frac{1}{2^{n_1}}\r)\r) 
\\
& = ( |x-2^{-n_1}|^2 + 2\sqrt{2} | x-2^{-n_1} || y-2^{-n_1} | \notag \\
& \hspace{5cm} + 2| y-2^{-n_1} |^2 )^{1/2}\notag \\
& \leq ( \sqrt{2}|x-2^{-n_1}|^2 + 2\sqrt{2} | x-2^{-n_1} || y-2^{-n_1} | \notag \\
& \hspace{5cm} + 2\sqrt{2} | y-2^{-n_1} |^2 )^{1/2}\notag \\
& = 2^{\frac{1}{4}} d(f^{-1}_{v}(x),f^{-1}_{v}(y)). \notag
\end{align} 
On the other hand, there exist $ x_0 \in [\frac{1}{2^{n_1}}-\delta,\frac{1}{2^{n_1}} )$ and $y_0 \in (\frac{1}{2^{n_1}},\frac{1}{2^{n_1}} +\delta ]$ such that
\begin{align*}
&d\l(f^{-1} \circ f^{-1}_{v}(x_0),f^{-1}_{u}\l(\frac{5}{2^{n+2}}\r) \r) 
\\
&= d\l(f^{-1} \circ f^{-1}_{v}(y_0),f^{-1}_{u}\l(\frac{5}{2^{n+2}}\r) \r).
\end{align*}
Thus  the triangle determined by the three vertices $f\circ f^{-1}_{v}(x_0)$, $f\circ f^{-1}_{v}(y_0)$ and $f^{-1}_{u}(5/2^{n+2})$ is an isosceles right triangle, and we can calculate
\begin{align} \label{eq: triangle}
&d\l(f^{-1} \circ f^{-1}_{v}(x_0),f^{-1} \circ f^{-1}_{v}(y_0) \r)\\ \notag
&=\frac{1}{\sqrt{2}}d\l(f^{-1} \circ f^{-1}_{v}(x_0),f^{-1}_{u}\l(\frac{5}{2^{n+2}}\r) \r) \\ \notag
& \quad + \frac{1}{\sqrt{2}}d\l(f^{-1} \circ f^{-1}_{v}(y_0),f^{-1}_{u}\l(\frac{5}{2^{n+2}} \r)\r).
\end{align}
By \eqref{eq: step2}, \eqref{ineq: triangle1} and \eqref{eq: triangle}, we have a bound for the dilation of $f$:
\begin{align} \label{ineq: dilation}
\begin{split}
\frac{1}{\mbox{dil}(f)} & \leq  \frac{ d\Bigl(f^{-1} \circ f^{-1}_{v}(x_0),f^{-1} \circ f^{-1}_{v}(y_0)\Bigr)}{d\Bigl(f^{-1}_{v}(x_0),f^{-1}_{v}(y_0)\Bigr)} \\
& =  \frac{d\Bigl(f^{-1} \circ f^{-1}_{v}(x_0),f^{-1}_{u}(\frac{5}{2^{n+2}}) \Bigr)}{\sqrt{2}d(f^{-1}_{v}(x_0),f^{-1}_{v}(y_0))}
\\
& \quad +\frac{d\Bigl(f^{-1} \circ f^{-1}_{v}(y_0),f^{-1}_{u}(\frac{5}{2^{n+2}}) \Bigr)}{\sqrt{2}d(f^{-1}_{v}(x_0),f^{-1}_{v}(y_0))}
\\
 \leq &\frac{\mbox{dil}(f^{-1})d\bigl( f^{-1}_{v}(x_0),f^{-1}_{v}(\frac{1}{2^{n_1}}) \bigr)}{\sqrt{2}d(f^{-1}_{v}(x_0),f^{-1}_{v}(y_0))}
 \\
 & \quad + \frac{\mbox{dil}(f^{-1}) d\bigl(f^{-1}_{v}(y_0),f^{-1}_{v}(\frac{1}{2^{n_1}}) \bigr)}{\sqrt{2}d(f^{-1}_{v}(x_0),f^{-1}_{v}(y_0))}
 \\
\le & \frac{1}{2^{\frac{1}{4}}}\mbox{dil}(f^{-1}).
\end{split}
\end{align}
Here we used the equality \eqref{eq: triangle} in the second line, the equality \eqref{eq: step2} and the definition of the dilation in the third line, and the inequality \eqref{ineq: triangle1} in the last line.
The inequality \eqref{ineq: dilation} implies that $\mbox{dil}(f)\geq 2^{\frac{1}{8}}$ or $\mbox{dil}(f^{-1}) \geq 2^{\frac{1}{8}}$. Thus we have
\[d_L(Y_{u},Y_{v}) \geq \frac{\log 2}{8}.\]
This contradicts the inequality \eqref{dil-2}.
Therefore we have, for any $n \in P_u$,
\[f_{v}\circ f \circ f^{-1}_{u}\l ( \frac{5}{2^{n+2}}\r) \notin \l\{ \frac{1}{2^{m}} :m \in \Z_{>0} \r\}.\]
By the same argument as above, we have, for any $n \in P_u$,
\[f_{v}\circ f \circ f^{-1}_{u}\l ( \frac{5}{2^{n+2}}\r) \notin  \l\{ \frac{3}{2^{m+1}} :m \in \Z_{>0}\r\}.\]

Now we start to show the third step (iii).
By the above two steps (i) and (ii), we have that, for any $n \in P_u$, there exists $p_f(n) \in \Z_{>0}$ such that 
\[f_{v}\circ f \circ f^{-1}_{u}\l ( \frac{5}{2^{n+2}}\r) =\frac{5}{2^{p_f(n)+2}} .\]
By the same argument as the first step (i), we can check that $p_f(n) \in P_v$, that is, $v_{p_f(n)}=1$.
Also for the inverse function $f^{-1}$, we have that,  
for any $n \in P_v$, there exists $p_{f^{-1}}(n) \in P_u$ such that  
\[f_{u} \circ f^{-1} \circ f^{-1}_{v}\l ( \frac{5}{2^{n+2}}\r) =\frac{5}{2^{p_{f^{-1}}(n)+2}}.\]
Since $f$ is a bijection, 
the map $p_f$ is a bijection from $P_u$ to $P_v$ and $p^{-1}_f=p_{f^{-1}}$.

Now it suffices to show that $p_f(n)=n$ for all $n \in P_u$. We assume that there exists $l\in P_u$ such that $p_f(l) \neq l$.
Without loss of generality, we may assume $p_f(l)>l$.
We first show that
\begin{align} \label{incl2-0} 
f_{u} \circ f^{-1} \circ f^{-1}_{v}\l(\frac{1}{2^{p_f(l)}}\r) \in \l(\frac{1}{2^{l}},\frac{5}{2^{l+2}}\r) \cup \l(\frac{5}{2^{l+2}},\frac{3}{2^{l+1}}\r).
\end{align}
To show this, it suffices to show that 
$$\frac{\sqrt{2}}{2^{l+2}} > d\l(f^{-1} \circ f^{-1}_{v}\l(\frac{1}{2^{p_f(l)}}\r),f^{-1}_{u}\l(\frac{5}{2^{l+2}}\r)\r),$$
where the above inequality means that the point $f^{-1} \circ f^{-1}_{v}(\frac{1}{2^{p_f(l)}})$ belongs to one of two edges in the $l$-pulse crossing at the right angle.
By $p_{f^{-1}}\circ p_f(l)=p^{-1}_{f}\circ p_f(l)=l$, we have
\[\begin{split} 
\frac{\sqrt{2}\mbox{dil}(f^{-1})}{2^{p_f(l)+2}} =& \mbox{dil}(f^{-1})d\l(f^{-1}_{v}\l(\frac{1}{2^{p_f(l)}}\r),f^{-1}_{v}\l(\frac{5}{2^{p_f(l)+2}}\r)\r)\\
\ge & d\l( f^{-1} \circ f^{-1}_{v}\l(\frac{1}{2^{p_f(l)}}\r),f^{-1}\circ f^{-1}_{v}\l(\frac{5}{2^{p_f(l)+2}}\r)\r)
\\
= & d\l( f^{-1} \circ f^{-1}_{v}\l(\frac{1}{2^{p_f(l)}}\r),f^{-1}_{u}\l(\frac{5}{2^{p_{f^{-1}}\circ p_f(l)+2}}\r)\r).
\\
= & d\l( f^{-1} \circ f^{-1}_{v}\l(\frac{1}{2^{p_f(l)}}\r),f^{-1}_{u}\l(\frac{5}{2^{l+2}}\r)\r).
\end{split}\]
Since we have $\mbox{dil}(f^{-1})\leq 2^{\frac{1}{4}}$ (by the inequality \eqref{dil-2}) and $p_f(l) \ge l+1$, it holds  
$$\frac{\sqrt{2}}{2^{l+2}}  > \frac{\sqrt{2}\mbox{dil}(f^{-1})}{2^{p_f(l)+2}} \ge d\l(f^{-1} \circ f^{-1}_{v}\l(\frac{1}{2^{p_f(l)}}\r),f^{-1}_{u}\l(\frac{5}{2^{l+2}}\r)\r).$$ 
Thus we have shown \eqref{incl2-0}.

By the continuity of $f$ and \eqref{incl2-0}, there exists $\delta>0$ such that $\delta< \frac{1}{2^{p_f(l)+3}}$ and  
\begin{align*}
 & f_{u} \circ f^{-1} \circ f^{-1}_{v}\l( \l[\frac{1}{2^{p_f(l)}}-\delta,\frac{1}{2^{p_f(l)}}+\delta\r]\r) 
 \\
 & \subset \l(\frac{1}{2^{l}},\frac{5}{2^{l+2}}\r) \cup \l(\frac{5}{2^{l+2}},\frac{3}{2^{l+1}}\r). 
 \end{align*}
Since the three points $f^{-1} \circ f^{-1}_{v}\l( \frac{1}{2^{p_f(l)}}-\delta \r), f^{-1} \circ f^{-1}_{v}\l( \frac{1}{2^{p_f(l)}} \r)$ and $ f^{-1} \circ f^{-1}_{v}\l( \frac{1}{2^{p_f(l)}}+\delta \r)$ are on the same line, we have 
\begin{align} \label{eq-2-1}
\begin{split}
& d\l(f^{-1} \circ f^{-1}_{v}\l( \frac{1}{2^{p_f(l)}}-\delta \r),f^{-1} \circ f^{-1}_{v}\l( \frac{1}{2^{p_f(l)}}+\delta \r)\r)\\
=& d\l(f^{-1} \circ f^{-1}_{v}\l( \frac{1}{2^{p_f(l)}}-\delta \r),  f^{-1} \circ f^{-1}_{v}\l( \frac{1}{2^{p_f(l)}} \r)\r)\\
&+d\l(f^{-1} \circ f^{-1}_{v}\l( \frac{1}{2^{p_f(l)}}\r),f^{-1} \circ f^{-1}_{v}\l( \frac{1}{2^{p_f(l)}}+\delta \r)\r).
\end{split}
\end{align}
Thus the inclusion \eqref{incl2-0} and the equality \eqref{eq-2-1} imply the following bound of the dilation of $f$:
\begin{align*}
\begin{split} 
& \mbox{dil}(f^{-1})
\\ \geq & \frac{d \l(f^{-1} \circ f^{-1}_{v}\l(\frac{1}{2^{p_f(l)}}-\delta\r),f^{-1} \circ f^{-1}_{v}\l(\frac{1}{2^{p_f(l)}}+\delta\r) \r)}{d\l( f^{-1}_{v}\l(\frac{1}{2^{p_f(l)}}-\delta\r), f^{-1}_{v}\l(\frac{1}{2^{p_f(l)}}+\delta\r) \r)}\\
 =& \frac{d \l(f^{-1} \circ f^{-1}_{v}\l(\frac{1}{2^{p_f(l)}}-\delta\r),f^{-1} \circ f^{-1}_{v}\l(\frac{1}{2^{p_f(l)}}\r) \r) }{\sqrt{5}\delta}\\
&+\frac{d \l(f^{-1} \circ f^{-1}_{v}\l(\frac{1}{2^{p_f(l)}}\r),f^{-1} \circ f^{-1}_{v}\l(\frac{1}{2^{p_f(l)}}+\delta\r) \r)}{\sqrt{5}\delta}\\
\geq & \frac{d \l( f^{-1}_{v}\l(\frac{1}{2^{p_f(l)}}-\delta\r), f^{-1}_{v}\l(\frac{1}{2^{p_f(l)}}\r) \r) }{\sqrt{5}\delta \mbox{ dil}(f)}
\\
& \quad + \frac{d \l( f^{-1}_{v}\l(\frac{1}{2^{p_f(l)}}\r), f^{-1}_{v}\l(\frac{1}{2^{p_f(l)}}+\delta\r) \r) }{\sqrt{5}\delta \mbox{ dil}(f)}\\
 =& \frac{\sqrt{2}+1}{\sqrt{5}\mbox{dil}(f)}.
\end{split}
\end{align*}
Here we used the following equality in the second line:  
\begin{align*}
\begin{split}
& d\l( f^{-1}_{v}\l(\frac{1}{2^{p_f(l)}}-\delta\r), f^{-1}_{v}\l(\frac{1}{2^{p_f(l)}}+\delta\r)\r)
\\
&=\l(\l|\frac{1}{2^{p_f(l)}}-\delta-\l(\frac{1}{2^{p_f(l)}}+\delta\r)\r|^2+|\delta|^2\r)^{1/2}\\
&=\sqrt{5}\delta.
\end{split}
\end{align*}
Thus $\mbox{dil}(f) \geq \l(\frac{\sqrt{2}+1}{\sqrt{5}}\r)^{1/2}$ or $\mbox{dil}(f^{-1}) \geq \l(\frac{\sqrt{2}+1}{\sqrt{5}}\r)^{1/2}$.
This contradicts the inequality \eqref{dil-2}.
Therefore we have $p_f(n)=n$ for any $n \in P_u$. 

We have completed all of the three steps. Setting $F(n)=p_f(n)$,  we have that the map $F: P_u \to P_v$ is a bijection such that $F(n)=n$ and this implies $P_u=P_v$. We have completed the proof.

\end{proof}

We resume the proof of Theorem \ref{thm: Int}.

{\it Proof of Theorem \ref{thm: Int}.} 
By using Lemma \ref{CE2}, we know that $(\mathbb Y, d_L)$ is discrete. 
Since the cardinality of $\mathbb Y$ is continuum and $\mathbb Y \subset \mathcal M_{[0,1]}$, we have that $(\mathcal M_{[0,1]}, d_L)$ is not separable. We have completed the proof.
\qed

\begin{remark} \label{rem: Ref}
Theorem \ref{thm: Int} states that $\mathcal M_{[0,1]}=\{X \in \mathcal M: d([0,1],X)<\infty\}$ is not separable. By the proof of Theorem \ref{thm: Int}, moreover we know the following stronger result, that is, the non-separability holds locally:

{\it Let $B_{d_L}([0,1], \delta)$ denote the ball in $\mathcal M_{[0,1]}$ centered at $[0,1]$ with radius $\delta>0$ with respect to the Lipschitz distance $d_L$, that is, 
$$B_{d_L}([0,1], \delta)=\{X \in \mathcal M_{[0,1]}: d_L([0,1], X)<\delta\}.$$
Then, for any $\delta>0$, $B([0,1], \delta)$ is not separable.
}

In fact, let \begin{align*}
J^{\epsilon}(n, 1)=&[3/2^{n+1}, 1/2^{n-1}] \times \{ 0 \} 
\\
&\cup \{ (x, \epsilon (3/2^{n+1}-x): 5/2^{n+2}\le x \le 3/2^{n+1} \} \\
& \cup \{ (x, \epsilon (x-1/2^n)): 1/2^n \le x \le 5/2^{n+2} \},
\\
J^\epsilon(n,0)=&J(n,0),
\\
Y^{\epsilon}_{u}=&\{(0, 0)\} \cup \bigcup_{n=1}^{\infty}J^{\epsilon}(n, u_n), \quad u=(u_n)_{n \in \Z_{>0}} \in \{0,1\}^{\Z_{>0}}.
\end{align*}
Then, by the similar proof to that of Theorem \ref{thm: Int}, we obtain
\begin{description}
	\item[(i)] For every $\epsilon>0$, the set
 
\[\mathbb Y^{\epsilon}=\{Y^{\epsilon}_u: u \in \{0, 1\}^{\Z_{>0}}\} /\mbox{isometry}\]
is discrete with cardinality of the continuum. 
	\item[(ii)] For every $\delta >0$,  there exists $\epsilon>0$ such that $\mathbb Y^\epsilon \subset B_{d_L}([0,1], \delta)$.
	\end{description}
The statement (ii) implies that $B_{d_L}([0,1], \delta)$ is not separable for any $\delta>0$.
\end{remark}
\section*{Acknowledgment}
We would like to thank an anonymous referee for careful reading of our manuscript and pointing out Remark \ref{rem: Ref}.
The first author was supported by Grant-in-Aid for JSPS Fellows Number 261798 and DAAD PAJAKO Number 57059240.

\end{document}